\documentclass[10pt]{amsart}
\usepackage{mathrsfs}
\usepackage{amsmath}
\usepackage{amssymb}
\usepackage{amsbsy}
\usepackage{graphicx}
\usepackage{amsfonts}

\newtheorem{lem}{Lemma}[section]%
\newtheorem{theorem}[lem]{Theorem}%
\newtheorem{prop}[lem]{Proposition}%

\def\a{\alpha}

\def\nd{\mathrel{\bigm|\kern-.7em/}}

\def\B{\mathcal{B}}

\def\Aut{\hbox{\rm Aut}}

\def\mz{{\mathbb Z}}

\def\authorsaddresses#1{\dedicatory{#1}}
\begin{document}

\title{Tetravalent $s$-transitive graphs of order $6p^2$}

\author[ M. Ghasemi, A.A.  Talebi and    N. Mehdipoor]{ M. Ghasemi$^1$,   A.A.  Talebi $^2$ and N. Mehdipoor$^3$    }

\authorsaddresses {$^1$ Department of Mathematics, Urmia University, Urmia 57135, Iran.  email:  m.ghasemi@urmia.ac.ir\\
$^2$ Department of Mathematics, Mazandaran University, email:  a.talebi@umz.ac.ir
 \\$^3$  Department of Mathematics, Mazandaran University, email: n.mehdipour@umz.ac.ir
  }

\subjclass[2000]{20B25, 05C25} \keywords {$s$-transitive graph,
Automorphism group, Cayley graph, Covering projection, Solvable groups.}

\begin{abstract}
Let $s$ be a positive integer. A graph is $s$-transitive if its automorphism
group is transitive on s-arcs but not on $(s + 1)$-arcs.  In this paper, we study  all tetravalent  $s$-transitive graphs of order $6p^2$.
\end{abstract}
\maketitle
\section{Introduction}
In this study, all graphs considered are assumed  to be
finite, simple and connected. For a graph $X$,
 $V(X)$, $E(X)$, $A(X)$ and  $\rm {Aut(X)}$  denote its vertex set, edge
set, arc set,  and full automorphism group, respectively.  For $u, v  \in  V(X)$,  $ \{u,v\} $ denotes  the
edge incident to $ u $ and $ v$ in $X$, and  $ N_{X}(u)$ denotes the
neighborhood of $u $ in $X$, that is, the set of vertices adjacent to
$ u $ in $X$.

An $s$-arc in a graph $X$ is an ordered $(s
+1)$-tuple ($v_{0} $, $v_{1} $, . . . , $v_{s} $) of vertices of $X$
such that $v_{i-1 }$ is adjacent to $v_{i}$ for $1  \leq  i \leq
s$, and $v_{i-1}  \neq v_{i+1}$ for $1  \leq  i  <  s$; in
the other words, a directed walk of length $s$ that never includes a
backtracking. For a graph $X$ and a subgroup $G$ of $\rm Aut(X)$, $X$ is said to
be $G$-vertex-transitive, $G$-edge-transitive, or $G$-$s$-arc-transitive if $G$
is transitive on the sets of vertices, edges, or $s$-arcs of $ X$,
respectively, and $G$-$s$-regular if $G$ acts regularly on the set of
$s$-arcs of $X$.  A graph $X$ is called vertex-transitive, edge-transitive, $s$-arc-transitive, or $s$-regular
if $X$ is $\rm Aut(X)$-vertex-transitive, $\rm Aut(X)$-edge-transitive,
$\rm Aut(X)$-$s$-arc-transitive, or $\rm Aut(X)$-$s$-regular, respectively. In
particular, $1$-arc-transitive means arc-transitive, or symmetric.
A $G$-$s$-arc-transitive graph is said to be $G$-$s$-transitive if it is not $G$-$(s+1)$-arc-transitive.

Symmetric or  s-transitive graphs with small valencies have received considerable
attention in the algebraic graph theory. Tutte \cite{Tut1} initiated the investigation
of cubic s-transitive graphs by proving that there exist no cubic $s$-transitive graphs
for $s \geq 6$. Gardiner and Praeger \cite{GP,GP1} generally explored the tetravalent arc-transitive graphs.
Let $1\leq k \leq 3$ be a integer.  The classification of s-transitive graphs
of order $kp$ and of valency 3 or 4 can be obtained from \cite{C,CO,WX1}.
Feng et al.  classified cubic s-transitive graphs of order $np$ or $np^{2}$ for $n=4,6,8$ or $10$ in \cite{FK2,FK3, FK4, FKW}.  Zhou
 and  Feng  studied tetravalent s-transitive graphs of order twice a prime power in \cite{JXZYQF}.  Feng et al. \cite{FK1} studied one-regular cubic graphs of order a small number times a prime or a prime square. Tetravalent graphs of order $6p^2$, admitting a group of
automorphisms acting regularity on arcs was investigated by  Ghasemi and Spiga in \cite{Gh.S}.     Zhou et al.  \cite{Gh,GZ} classified tetravalent s-transitive graphs of order $3p^{2}$ and $4p^{2}$.  Zhou classified Tetravalent s-transitive graphs of order $4p$ in \cite{Z}.  Ghasemi and Varmazyar  classified  the  tetravalent arc-transitive graphs of order $5p^2$ in \cite{Gh.V}. The aim of this paper is to classify all tetravalent s-transitive graphs of order $6p^2$.

\section{Preliminaries}
In this section, we introduce some notations and definitions as well
as some preliminary results which will be used later in the paper.

For a regular graph $X$, use $d(X)$ to represent the valency of $X$,
and for any subset $B$ of $V(X)$, the subgraph of $X$ induced by $B$
will be denoted by $X[B]$. Let $X$ be a connected vertex-transitive
graph, and let $G\leq \Aut(X)$ be vertex-transitive on $X$. For a
$G$-invariant partition $\B$ of $V(X)$, the  quotient graph
$X_\B$ is defined as the graph with vertex set $\B$ such that, for
any two vertices $B,C\in \B$, $B$ is adjacent to $C$ if and only if
there exist $u\in B$ and $v\in C$ which are adjacent in $X$. Let $N$
be a normal subgroup of $G$. Then the set $\B$ of orbits of $N$ in
$V(X)$ is a $G$-invariant partition of $V(X)$. In this case, the
symbol $X_\B$ will be replaced by $X_N$.  If $X_N$ and $X$ have the same valency, then $X$ is
called a normal cover of $X_N$. For a positive integer $n$, denote by $\mz_n$ the cyclic group of
order $n$ as well as the ring of integers modulo $n$, by $D_{2n}$ the dihedral group of order $2n$, respectively.

Let $G$ be a permutation group on a set $\Omega$ and $\a\in \Omega$.
Denote by $G_\a$ the stabilizer of $\a$ in $G$, that is, the
subgroup of $G$ fixing the point $\a$. We say that $G$ is {\em
semiregular} on $\Omega$ if $G_\a=1$ for every $\a\in \Omega$ and
 regular if $G$ is transitive and semiregular.

In graph theory, the lexicographic product or (graph) composition $G[H]$ of graphs $G$ and $H$ is a graph such that the vertex set of $G[H]$ is the cartesian product $V(G)\times V(H)$; and
any two vertices $(x,y)$ and $(v,w)$ are adjacent in $G[H]$ if and only if either $x$ is adjacent with $v$ in $G$ or $v = x$ and $w$ is adjacent with $y$ in $H$. Clearly, if $G$ and $H$ are arc-transitive then $G[H]$ is arc-transitive.

Let $G$ be a group, and let $S \subseteq G$ be a set of group elements such that the identity element $1$ not in $S$. The Cayley graph associated with $(G,S)$ is  defined as the graph having one vertex associated with each group element,  edges $(g,h)$ whenever $hg^{-1}$ in $S$. The Cayley graph $X$ is denoted by $\rm {Cay(G, S)}$.
 The following proposition gives a characterization for
Cayley graphs in terms of their automorphism groups.

\begin{prop}{\rm \cite[Lemma~16.3]{Biggs}}\label{cayley graph}
A graph $X$ is isomorphic to a Cayley graph on a group $G$ if and
only if its automorphism group has a subgroup isomorphic to $G$,
acting regularly on the vertex set of $X$.
\end{prop}

A graph
$\widetilde{X}$ is called a  covering of a graph $X$ with
projection $p:\widetilde{X}\rightarrow X$ if there is a surjection
$p:V{(\widetilde{X})}\rightarrow V(X)$ such that
$p|_{N_{\widetilde{X}}({\tilde{v}})}:{N_{\widetilde{X}}({\tilde{v}})}\rightarrow
{N_{X}(v)}$ is a bijection for any vertex $v\in V(X)$ and
$\tilde{v}\in p^{-1}(v)$. A covering $\widetilde{X}$ of $X$ with a
projection $p$ is said to be  regular (or $K$-covering) if
  there is a
 subgroup $K$ of the automorphism group ${\rm Aut( \widetilde{X})}$  such that  $K$ is semiregular on both $V(\widetilde{X})$ and $E(\widetilde{X})$ and graph
 $X$ is isomorphic to the
quotient graph $\widetilde{X}/K$, say by $h$, and the quotient map
$\widetilde{X}\rightarrow \widetilde{X}/K$ is the composition $ph$
of $p$ and $h$ (for the purpose of this paper, all functions are
composed from left to right). If $K$ is cyclic or elementary abelian
then $\widetilde{X}$ is called a  cyclic or an  elementary
abelian covering of $X$. The group of covering transformations $\rm {CT(p)}$ of $p : \widetilde{X} \rightarrow X$ is the group of all self equivalences of $p$, that is, of all automorphisms $ \widetilde{\alpha } \in {\rm  Aut(  \widetilde{X})}$ such that $p= \widetilde{\alpha}  p$. If
$\widetilde{X}$ is connected, $K$ becomes the covering transformation
group. The  fibre of an
edge or a vertex is its preimage under $p$. An automorphism of
$\widetilde{X}$ is said to be fibre-preserving if it maps a
fibre to a fibre, while every covering transformation maps a fibre
on to itself. All of fibre-preserving automorphisms form a group
called the  fibre-preserving group.  If
$|fib_{u}| = n$, we say that the covering projection $p$ is n-fold or fold number is $n$.  It is clear any covering of a bipartite graph is bipartite. The next proposition is said when a bipartite graph is a covering of a non-bipartite graph.

\begin{prop}{\rm \cite[Corollary ~3.2]{bipartite}}\label{fold}
 If $\widetilde{X}$ is a bipartite covering of a non-bipartite graph $X$,
then the fold number is even.
\end{prop}

Let $\widetilde{X}$ be a $K$-covering of $X$ with a projection $p$.
If $\alpha$$\in$ Aut($X$) and $\widetilde{\alpha}$$\in$
Aut($\widetilde{X}$) satisfy $\widetilde{\alpha}p=p\alpha$, we call
$\widetilde{\alpha}$ a {\it lift} of $\alpha$, and $\alpha$ the {\it
projection} of $\widetilde{\alpha}$. Concepts such as a lift of a
subgroup of Aut($X$)  and the projection of a subgroup of
Aut$(\widetilde{X})$ are self-explanatory. The lifts and the
projections of such subgroups are of course subgroups in
Aut($\widetilde{X}$) and Aut($X$) respectively.

 A graph $X$ is called a
bi-Cayley graph over a group $H$ if it has a semiregular automorphism group, isomorphic
to $H$, which has two orbits in the vertex set. Given such $X$, there exist subsets $R$, $L$, $S$ of $H$ such that $R^{-1} = R$, $L^{-1} = L$, $1 \notin R \cup L$, and $X = \rm{BiCay}(H, R, L, S)$, where the latter graph is defined to have vertex set the union the right part $H_{0} = \{ h_{0} | h \in H \}$
and the left part $H_{1} = \{ h_{1} | h \in H \}$; and the edge set consists of three sets:
\begin{flushleft}
$\{ \{h_{0}, g_{0}\} | gh^{-1} \in R \}$ (right edges),\\
$\{ \{h_{1}, g_{1}\} | gh^{-1} \in L \}$  (left edges),\\
$\{ \{h_{0}, g_{1}\} | gh^{-1} \in S \}$ (spoke edges).
\end{flushleft}

If $|R| = |L| = s$, then $\rm{BiCay}(H, R, L, S)$ is said to be
an $s$-type bi-Cayley graph, and if $|L| = |R| = 0$, then $\rm{BiCay}(H, S)$ will be written for
$\rm{BiCay}(H, \emptyset, \emptyset, S)$.  Such a graph $X$ is called normal edge-transitive if the normaliser of $H$ in the full automorphism group of $X$ is transitive on the edges of $X$.
\begin{prop}[{\cite{conder1}, Lemma 3.1}]\label{bipartite}
Let $X = \rm{BiCay}(H, R, L, S)$ be a connected normal edge-transitive bi-Cayley graph over the
group $H$. Then $R = L = \emptyset$, and hence $X$ is bipartite, with the two orbits of $R(H)$ on $V (X)$ as its parts.

\end{prop}

The following proposition is due to Praeger et al, refer to
[\cite{GP}, Theorem~1.1] and \cite{Praeger}.

\begin{prop}\label{reduction}
Let $X$ be a connected tetravalent $(G,1)$-arc-transitive graph. For
each normal subgroup $N$ of $G$, one of the following holds:
\begin{enumerate}
\item [{\rm (1)}] $N$ is transitive on $V(X)$;

\item [{\rm (2)}] $X$ is bipartite and $N$ acts transitively on
each part of the bipartition;

\item [{\rm (3)}] $N$ has $r\geq 3$ orbits on $V(X)$, the quotient
graph $X_N$ is a cycle of length $r$, and $G$ induces the full
automorphism group $D_{2r}$ on $X_N$;

\item [{\rm (4)}] $N$ has $r\geq 5$ orbits on $V(X)$, $N$ acts
semiregularly on $V(X)$, the quotient graph $X_N$ is a connected
tetravalent $G/N$-symmetric graph, and $X$ is a $G$-normal cover of
$X_N$.
\end{enumerate}
Moreover, if $X$ is also $(G,2)$-arc-transitive, then case $(3)$ can
not happen.
\end{prop}

The following proposition characterizes the vertex stabilizer of the
connected tetravalent $s$-transitive graphs, which can be deduced
from \cite[Lemma~2.5]{LPZ}, or \cite[Proposition~2.8]{LLM}, or
\cite[Theorem~2.2]{Li-TAMS2001}.

\begin{prop}\label{4-val-sym}
Let $X$ be a connected tetravalent $(G, s)$-transitive graph. Let
$G_v$ be the stabilizer of a vertex $v\in V(X)$ in $G$. Then
$s=1,2,3,4$ or $7$. Furthermore,  $G_v$ is a $2$-group for $s=1$;
$G_v$ is isomorphic to $A_4$ or $S_4$ for $s=2$;  $G_v$ is
isomorphic to $A_4\times\mz_3$, $\mz_3\times S_4$, or $S_3\times
S_4$ for $s=3$;  $G_v$ is isomorphic to $\mz_3^2\rtimes{\rm
GL(2,3)}$ for $s=4$; and $G_v$ is isomorphic to $[3^5]\rtimes{\rm
GL(2,3)}$ for $s=7$, where $[3^5]$ represents an arbitrary group of
order $3^5$.
\end{prop}
For a subgroup $H$ of a group $G$, denote by $C_{G}(H)$  the centralizer of  $H$  in $G$  and by $N_{G}(H)$ the normalizer of H in G.
\begin{prop}[\cite{H}, Chapter I, Theorem 4.5]\label{Normal}
 The quotient group $N_{G}(H)/C_{G}(H) $ is isomorphic to a subgroup of the
automorphism group $\rm Aut(H)$ of H.
\end{prop}
The
following proposition is said    a result of the well-known classification of finite simple groups.
\begin{prop}\label{4}
\emph{\cite{G}} A non-abelian simple group whose order has at
most three prime divisors is isomorphic to one of the following groups:
\begin{center}
$\rm {A_{5}, A_{6},PSL(2,7),PSL(2,8),PSL(2,17),PSL(3,3),PSU(3,3),PSU(4,2)}$,
 \end{center}
 whose orders are $2^{2}~.~ 3 ~. ~5$, $ 2^{3} ~.~ 3^{2} ~.~ 5$, $ 2^{3} ~.~ 3 ~.~ 7$, $ 2^{3} ~.~ 3^{2} ~.~ 7$, $2^{4} ~.~ 3^{2} ~.~ 17$, $ 2^{4} ~.~ 3^{3} ~.~ 13$, $
2^{5}~.~3^{3}~.~7$, $ 2^{6} ~.~3^{4}~.~5$, respectively.
\end{prop}
\begin{prop}[\cite{Gh.S}, Theorem 1.1]\label{pgh}
Let $p$ be a prime and let $X$ be a tetravalent graph of order $6p^2$
admitting a group of automorphisms acting regularly on $A(X)$. Then
one of the following holds:
\begin{description}
\item[(i)]$X$ is isomorphic to $C(2;3p^2,1)$, $C^{\pm 1}(p;6,2)$, $Y_{p,\pm 1}$, $Y_{p,\pm \sqrt{3}}$, $Z_{p,\pm\sqrt{-1}}$ or $Z_{p,\pm \sqrt{-3}}$. (see Section 3 for the definition of these graphs);
\item[(ii)] $X$ is a Cayley graph over $G$ with connection set $S$ where
\begin{description}
\item[(a)]$G=\langle x,y\mid x^{p}=y^{6p}=[x,y]=1\rangle$ and $S=\{y,y^{-1},xy,(xy)^{-1}\}$, or
\item[(b)]$G=\langle x,y,z\mid x^p=y^{3p}=z^2=[x,y]=[x,z]=1,y^z=y^{-1}\rangle$ and $S=\{xz,x^{-1}z,x^\varepsilon yz,x^{-\varepsilon}yz\}$ (here $\varepsilon^2\equiv -1\pmod p$ and $p\equiv 1\pmod 4$), or
\item[(c)]$G=\langle x,y,z\mid x^{p^2}=y^{3}=z^2=[x,y]=[x,z]=1,y^z=y^{-1}\rangle$ and
$S=\{xz,x^{-1}z,xyz,x^{-1}yz\}$, or
\item[(d)]$G=\langle x,y,z\mid x^{p^2}=y^{3}=z^2=[x,y]=[x,z]=1,y^z=y^{-1}\rangle$ and $S=\{xz,x^{-1}z,x^\varepsilon yz,x^{-\varepsilon}yz\}$  (here $\varepsilon^2\equiv -1\pmod {p^2}$), or
\item[(e)]$G=\langle x,y,z,t\mid x^p=y^{p}=z^3=t^2=[x,y]=[x,z]=[x,t]=[y,z]=[y,t]=1,z^t=z^{-1}\rangle$ and $S=\{xt,x^{-1}t,yzt,y^{-1}zt\}$;
\end{description}
\item[(iii)]$p\in \{2,3,5\}$ and $X$ is described in Section 6.
\end{description}
\end{prop}

Let $X$ be a graph and $K$ be a finite group. By $a^{-1}$ we mean the
reverse arc to an arc $a$. A voltage assignment (or $K$-voltage
assignment) of $X$ is a function $\xi : A(X) \rightarrow  K$ with the
property that $\xi(a^{-1}) = \xi(a)^{-1}$ for each arc $a
\in A(X)$. The values of $\xi$ are called voltages, and $K$ is the
voltage group. The graph $ X  \times_{{\xi}} K$  derived from a
voltage assignment $\xi: A(X)\rightarrow K$ has vertex set $V(X)
\times K$ and edge set $E(X) \times K$, so that an edge $(e, g)$ of $X$
$\times$ K joins a vertex $(u, g)$ to $(v, \xi(a)g)$ for $a = (u, v)
\in A(X)$ and $g \in K$, where $e = \{u,v\}$. Clearly, the derived graph $ X \times_{{\xi}} K$ is a covering of $X$
with the first coordinate projection $ p :  X \times_{{\xi}} K
\rightarrow  X$, which is called the natural projection. By defining
$(u, g')^{g} = (u, g'g)$ for any $g \in K$ and $(u, g') \in
V(X \times_ {\xi} K)$, $K$ becomes a subgroup of $ {\rm Aut(X \times _{\xi} K)}$ which acts semiregularly on $V( X
\times_{{\xi}} K)$. Therefore, $ X \times_{{\xi}} K$ can be viewed
as a $K$-covering. For each $u \in V(X)$ and $\{u,v\} \in E(X)$, the vertex
set $\{(u, g) | g\in K\}$ is the fibre of $u$ and the edge set $\{(u,
g) (v,\xi(a)g) | g\in K\}$ is the fibre of $\{u,v\}$, where $a = (u, v)$. The group $K$ of automorphisms of $X$ fixing every fibre setwise
is  the covering transformation group.
Conversely, each regular covering $\widetilde{X}$ of $X$ with a
covering transformation group $K$ can be derived from a $K$-voltage
assignment.  Given a spanning tree $T$ of the graph $X$, a voltage assignment
$\xi$ is said to be $T$-reduced if the voltages on the tree arcs are
the identity. Gross and Tucker in \cite{Gross} showed that every regular
covering $\widetilde{X}$ of a graph $X$ can be derived from a $T$-reduced voltage assignment $\widetilde{X}$ with respect to an
arbitrary fixed spanning tree $T$ of $X$.

\section{classification}
Let $X$ be a connected tetravalent
$s$-transitive graph  of order $6p^2$, where $p$ is prime. By  \cite{PSV}, we may suppose that $p\geq 11$.

\begin{lem}\label{10}
Let $X$ be a  edge-transitive graph,  $p$ is a prime and $N  \trianglelefteq \rm {Aut(X)}$, where $N \cong \mathbb Z_{p}$. If the quotient graph $X_{N}$ is a normal Cayley graph and has  the same valency with $X$ then $X$ is a $N$-regular covering of $X_{N}$ and $X$ is a normal Cayley graph.
\end{lem}
\begin{proof}
Let $N$ be a normal subgroup of $A:=\rm {Aut(X)}$ and  $X_{N}$ be the quotient graph of $X$ with respect to the orbits of $N$ on $V(X)$. Assume that  $K$ is the kernel of $A$ acting on $V(X_{N})$. The stabilizer $K_{v}$ of $v \in V(X)$ in $K$ fixes the neighborhood of $v$ in $X$. The
connectivity of $X$ implies $K_{v}=1$ for any $v \in V(X)$ and hence  $N_{v}=1$.  If $N_{\{ \alpha,\beta\}} \neq 1$ then $N_{\{ \alpha,\beta\}}=N$. Since $X$ is  connected, there is a $\{ \beta, \gamma\} \in E(X)$ where $\beta,\gamma \in V(X)$. Then we have $g \in A$ such that $\{ \alpha,\beta\}=\{ \beta, \gamma\}^{g}$ because $X$ is an edge-transitive graph.  Hence $N_{\{ \alpha,\beta\}}=N_{\{ \beta, \gamma\}^g}= g^{-1}N_{\{ \beta, \gamma\}}g= N_{\{ \beta, \gamma\}}$. It is a contradiction and so $N_{\{ \alpha,\beta\}}=1$. Therefore $X$  is a $\mathbb Z_{p}$-regular covering of $X_{N}$. Now we prove that $X$ is a Cayley graph.
Let $X_{N}\cong \rm{Cay(G,S)}$, $X\cong X_{N} \times_{\xi}\mathbb Z_{p}$ where $\xi$ is the $T$-reduced voltage assignment  and $\tilde{G}$ is a lift of $G$ such that $\tilde{\alpha}p=p\alpha$ where $p: X \rightarrow X_{N}$ is regular covering projection, $\alpha \in \rm {Aut(X_{N})}$ and $\tilde{\alpha} \in A$.  For any $(x,k) , (y,k') \in V(X)$ where $k, k' \in \mathbb Z_{p}$ and $x,y \in V(X_{N})$, we have $\alpha \in \rm {Aut(X_{N})}$ such that $x^{\alpha}=y$. For $k^{''} \in \mathbb Z_{p}$, $(x,k)^{\tilde{\alpha}p}=(z, k^{''})^{p}=z$ where $(x,k)^{\tilde{\alpha}}=(z, k^{''})$. Also $(x,k)^{p\alpha} =x^{\alpha}=y$. Then $y=z$ and hence $(y,k), (y, k^{''}) \in p^{-1}(y)$. Therefore $\tilde{G}$ is transitive on $V(X)$. Now, we prove that $\tilde{G}$ is semiregular. Suppose that $(x,k)^{\tilde{\alpha}}=(x,k)$.   Now, since $G$ is semiregular and $\tilde{\alpha}p=p\alpha$, it implies that $x=(x,k)^{\tilde{\alpha}p}=(x,k)^{p\alpha}=x^{\alpha}$. Then $\alpha=1$ and hence $\tilde{\alpha}p=p$. Therefore $\tilde{\alpha} \in \rm{CT(p)}=\mathbb Z_{p}$ and since $\rm{CT(p)}$ is semiregular, it  follows that $\tilde{\alpha}=1$.

Assume that $X_{N}$ is a normal  Cayley graph and $G \trianglelefteq \rm {Aut(X_{N})}$. Let  $\tau \in \rm {Aut(X_{N})}$ and $\tilde{\tau}$ is a lift $\tau$. First, we prove that $\tilde{\tau}^{-1}p=p\tau^{-1}$. Suppose that $\tilde{x}, \tilde{y} \in V(X)$ and $\tilde{x}\tilde{\tau}^{-1}=\tilde{y}$. Since $(\tilde{y})p\tau=(\tilde{y})\tilde{\tau}p= \tilde{x}p$, we have $\tilde{y}p=\tilde{x}p\tau^{-1}$ and hence $\tilde{x}\tilde{\tau}^{-1}p=\tilde{x}p\tau^{-1}$. Then $\tilde{\tau}^{-1}p=p\tau^{-1}$. Suppose that $g \in G$, $\tilde{g} \in \tilde{G}$ and $\tilde{g}$  is a lift of $g$. Then $\tilde{\tau}^{-1}\tilde{g}\tilde{\tau}p=\tilde{\tau}^{-1}\tilde{g}p\tau=\tilde{\tau}^{-1}pg\tau=p\tilde{\tau}^{-1}g\tau=
p\tilde{\tau}^{-1}g\tilde{\tau}$.  Thus $\tilde{\tau}^{-1}\tilde{g}\tilde{\tau} \in \tilde{G}$. Therefore $\tilde{G} \trianglelefteq \rm {Aut(X)}$ and hence $X$ is a normal Cayley graph.
\end{proof}
\begin{lem}\label{degree}
Let $X$ be a  connected graph and $N  \leq \rm {Aut(X)}$ and $p: X \rightarrow X/N$ be the
corresponding quotient projection (that is $p(x)= x^ N$ for $x \in V(X) \cup A(X)$). If $N$ is semiregular then for every $v \in V (X)$, the valency of $v$ equals the valency of $p(v)$.
\end{lem}
\begin{proof}
Suppose that $p: X \rightarrow X/N$ is the
corresponding quotient projection. If $|N_{X}(v)\cap O| = 0$ or $1$ for  any   $N$-orbit $O$  and  $ v \in V(X \setminus O)$ then valency of $v$ equals the valency of $p(v)$. Now assume that
$|N_{X}(v)\cap O| > 1$ and   $(v, v_{1}), (v, v_{2}) \in A(X)$ such that $v_{1}, v_{2} \in O$. Since  $p((v, v_{1}))=(v^N,v_{1}^N ), p((v, v_{2}))=(v^N,v_{2}^N )$ and $v_{1}^N=v_{2}^N$, we have $ (v, v_{1}), (v, v_{2})$ belong to a fiber and hence there is $n \in N $ such that $(v, v_{2})= (v, v_{1})^n$. Thus $v^n=v$ and  $v_{2}= v_{1}^n$. This is a contradiction because $N$ is semiregular  for every $v \in V (X)$.
\end{proof}
\begin{lem}\label{lemma-a}
Let $p$ be a prime, and let $X$ be a connected tetravalent
$s$-transitive graph  of order $6p^2$, where $p\geq 11$. If $G \leq {\rm Aut}(X)$ is transitive on the arc set of $X$, then every minimal normal subgroup of $G$ solvable.
\end{lem}
\begin{proof} Suppose that $G$ is arc-transitive on $X$ and $v\in V(X)$. By
Proposition~\ref{4-val-sym}, $G_v$ either is a $2$-group or has
order dividing $2^4\cdot 3^6$. It follows that $|G| \mid 2^5~.~3^7~.~p^2$ or $|G|=2^{m+1}~.~3~.~p^2$ for some integer $m$. Let $N$ be a minimal normal subgroup of $G$. Suppose that $N$ is non-solvable.  First suppose that $|G| \mid 2^{m+1}~.~3~.~p^2$. Also let $N \cong T_1 \times T_2 \cdots \times T_n$, where $T_i \cong T_j$ and $1\leq i, j \leq n$. Since $3 \mid |G|$ and $3 ^2 \nmid |G|$ then we conclude that $N \cong T$, where $T$ is isomorphic to one of the  groups in Proposition  \ref{4}. By considering the order of  these groups and since $p \geq 11$ we get a contradiction.
Now suppose that $|G| \mid 2^5~.~3^7~.~p^2$.  Thus $X$ is 2-arc-transitive. By considering the orders of  groups in Proposition  \ref{4},   we may suppose that $ N \cong$ PSL(2,17)  or PSL(3,3). Let $X_{N}$ be the quotient graph of $X$ with respect to the orbits of $N$ on $V(X)$. By Proposition \ref{reduction},  $d(X_{N})= 2$ or $4$.  If $d(X_{N})= 4$ then  $N_v=1$ for every $v \in V (X)$. By considering $|N|$ and $N$-orbit $|v^N|$, a contradiction can be obtained. If $d(X_{N})= 2$ then we  consider the    Proposition \ref{reduction}. Since   $X$ is  (G, 2)-arc-transitive, then case (3) of Proposition \ref{reduction} can not happen. Hence we get a contradiction.

\end{proof}

\begin{center}
${  \left(\begin{array}{cccccc}
  \hline
  Graph & order & s-transitive   \\
  \hline
  C[24,1] & 6~.~2^2 & 1 \\
  C[24,2] & 6~.~2^2 & 1 \\
  C[24,3] & 6~.~2^2 & 1 \\
 C[24,4] & 6~.~2^2 & 1 \\
 C[24,5] & 6~.~2^2 & 1 \\
 C[24,6] & 6~.~2^2 & 1 \\
 C[54,1] & 6~.~3^2 & 1 \\
  C[54,2] & 6~.~3^2 & 1 \\
  C[54,4] & 6~.~3^2 & 2 \\
  C[54,5] & 6~.~3^2 & 1 \\
  C[54,6] & 6~.~3^2 & 1 \\
  C[150,1] & 6~.~5^2 & 1 \\
  C[150,2] & 6~.~5^2 & 1 \\
  C[150,3] & 6~.~5^2 & 1 \\
  C[150,4] & 6~.~5^2 & 1 \\
  C[150,5] & 6~.~5^2 & 1 \\
  C[150,6] & 6~.~5^2 & 1 \\
  C[150,7] & 6~.~5^2 & 1 \\
  C[150,8] & 6~.~5^2 & 1 \\
  C[150,10] & 6~.~5^2 & 1 \\
  C[150,11] & 6~.~5^2 & 1 \\
  C[294,1] & 6~.~7^2 & 1 \\
  C[294,2] & 6~.~7^2 & 1 \\
  C[294,3] & 6~.~7^2 & 1 \\
  C[294,5] & 6~.~7^2 & 1 \\
  C[294,8] & 6~.~7^2 & 1 \\
  C[294,10] & 6~.~7^2 & 1 \\
  C[294,11] & 6~.~7^2 & 1 \\
  \hline
 \end{array}
 \right) }$\\
 Table 1: s-transitive graphs of order $6p^2$ with $p<11$ \end{center}
\begin{theorem}
Suppose that $X$ is a tetravalent $s$-transitive graph of order $6p^2$, where $p$ is a prime.  Then either $X$  is 1-regular or  $X$  is isomorphic to 1-transitive graph $C_{3p^2}[2K_1]$ or  it is isomorphic to one of the graphs in Table 1.\end{theorem}
\begin{proof}
 Suppose that $A={\rm Aut}(X)$ and   $N$ is a minimal normal subgroup of $A$. Also suppose that  $\Omega$ is the set of orbits of $N$ on $V(X)$  and $K$ be the kernel of the action $A$ on $\Omega$. If $X$ is $1$-regular then  $X$ is isomorphic to one of the graphs in  Proposition \ref{pgh}. Thus we may suppose that $X$ is not $1$-regular.  By Lemma \ref{lemma-a}, $N$ is an elementary abelian $r$-group, where $r \in \{2, 3, p\}$. First suppose that $N$ is an elementary abelian $2$-group and $\Omega=\{\Delta_0, \Delta_1, \Delta_2, \cdots  ,\Delta_{3p^2-1}\}$, where the subscripts are taken modulo $3p^2$.  Then $|X_N|=3p^2$ and by Proposition \ref{reduction}, $d(X_N)=2$ or $4$.  First suppose that $d(X_N)=2$. We know that   $\Delta_i$ does not have an edge then $X \cong C_{3p^2}[2K_1]$.
 Hence we may suppose that $d(X_N)=4$. Therefore $K=N \cong \mathbb{Z}_2$ and $X_N$ is $A/N$-symmetric graph. By [\cite{Gh}, Theorem 3.3],  $X_N$ is $1$-regular and so $|{\rm Aut}(X_N)|=2^2~.~3~.~p^2$. Thus $|A/K| \mid 2^2~.~3~.~p^2$ and hence  $|A| \mid 2^2~.~3~.~p^2$. Therefore  $X$ is $1$-regular, a contradiction. Now suppose that $N$ is an elementary abelian $3$-group and $\Omega=\{\Delta_0, \Delta_1, \Delta_2, \cdots  ,\Delta_{2p^2-1}\}$, where the subscripts are taken modulo $2p^2$.  Then $|X_N|=2p^2$ and by Proposition \ref{reduction}, $d(X_N)=2$ or $4$.  First suppose that $d(X_N)=2$.  We know that $\Delta_i$ does not have an edge then it is easy to see  that $K$ acts faithfully on  $\Delta_i$. Thus $|K| \leq S_3$ and since $A/K \cong D_{4p}$ we conclude that $|A|\leq 24p^2$. Therefore $X$ is $1$-regular, a contradiction. Thus we may suppose that $d(X_N)=4$ and so $K \cong  N \cong \mathbb{Z}_3$. Now by [\cite{JXZYQF}, Theorem 3.3]  either $|X_N|=8$ or $|{\rm Aut(X_N)}| \in \{2^{p^2+1}~.~p^2, 2^3~.~p^2, 2^4~.~p^2\}$. If $|X_N|=8$ then $p=2$, a contradiction. Also if  $|{\rm Aut(X_N)}|=2^3~.~p^2$  then $X$ is $1$-regular, a contradiction. Thus we may suppose that $|{\rm Aut(X_N)}|= 2^{p^2+1}~.~p^2$, or $2^4~.~p^2$.  If $|{\rm Aut(X_N)}|= 2^4~.~p^2$ then by [\cite{JXZYQF}, Proposition 2.5]  $X_N$ is normal Cayley graph on group $G\cong \mathbb{Z}_{2p} \times \mathbb{Z}_{p}$ of order $2p^2$. Also we know that $K=N$ and $A/K$ acts transitive on $V(X_N)$.  Thus $8p^2 \mid |A/K|$. If $|A/K|=8p^2$ then $|A|=24p^2$ and so $X$ is $1$-regular, a contradiction. Thus we may suppose that $|A/K|>8p^2$ and so $A/K={\rm Aut}(X_N)$.  Now since $G \leq A/K$ we conclude that $G=T/K$, where  $T \leq A$. Now $|T|=6p^2$ and $T$ acts regularly on $V(X)$. Now by Lemma \ref{10} $X$ is a normal Cayley graph on $T$, and hence   $T$ is generated by two  elements of same order. Since  $|T'|=1$ or $3$, by considering  the classification of groups of order $6p^2$ [\cite{Eick}, Table 4], $T \cong \mathbb{Z}_{p} \times \mathbb{Z}_{6p},~D_6 \times \mathbb{Z}_{p^{2}}$ or $\mathbb{Z}_{6p^2}$. By [\cite{Gh.S}, section 4.1,  4.5, 4.6], tetrvalent normal Cayley graphs over $T$ are 1-regular.  Also if  $|{\rm Aut(X_N)}|= 2^{p^2+1}~.~p^2$ then by  [\cite{JXZYQF}, Example 2.1], $X_N \cong C_{p^2}[2K_1]$ and $X_N={\rm Cay}(G, S)$, where $G=\mathbb{Z}_{p^2} \times \mathbb{Z}_{2}=\langle a \rangle \times \langle b \rangle$ and $S=\{a, a^{-1}, ab, a^{-1}b\}$.
 Since $A_v$ is $\{2, 3\}$-group we conclude that $A$ has a semiregular element of order $p^{2}$, say $\alpha$. Since $N_A(N)/C_A(N)$ is a subgroup of ${\rm Aut}(N) \cong \mathbb{Z}_{2}$ we conclude that $\alpha \in C_A(N)$. Also if $N=\langle \beta \rangle$ then $\langle \alpha, \beta \rangle$ is a semiregular subgroup of ${\rm Aut}(X)$ which  is isomorphic to $\mathbb{Z}_{3p^2}$. Thus  $X={\rm BiCay}(H,R,L,S)$,  for some suitable subsets $R$, $L$ and $S$ of  $H \cong \mathbb{Z}_{3p^2}=\langle c \rangle$.

We should note that $X$ is either $0$-type or $2$-type. First suppose that $X$ is $0$-type. Then by Proposition \ref{bipartite}, $X$ is a bipartite bi-Cayley graph on $S$.  Since $X$ is $N$-covering of $C_{p^2}[2K_1]$ and $C_{p^2}[2K_1]$ for $p >2$ is non-bipartite, we get a
contradiction by Proposition \ref{fold}.  Thus we may suppose that $X$ is $2$-type. Then $|R|=|L|=|S|=2$. Since, by [\cite{AT}, Lemma 3.1(3)], ${\rm BiCay}(H,R,L,S)\cong {\rm BiCay}(H,xR^\alpha,xL^\alpha,xS^\alpha)$ for each $x\in H$ and $\alpha\in\Aut(H)$, we may assume that $R=\{c,c^{-1}\}$, $L=\{c^i,c^{-i}\}$ and $S=\{1,c^j\}$ for some $1\leq i, j\leq 3p^2-1$, which means that $X$ is a rose-window graph. Since $X$ is a edge-transitive  graph, by \cite[Corollary 1.3]{KKM}, $X \cong R_{3p^2}(2, 1)$.  Also, by  section 3 of \cite{KKM},  this graph is isomorphic to $C_{3p^2}[2K_1]$.

  Now suppose that $N \cong \mathbb{Z}_p$ or $\mathbb{Z}_p \times \mathbb{Z}_p$. If  $N \cong \mathbb{Z}_p$ then $|X_N|=6p$. Since $N$ is a normal subgroup  of $A$,  we have   $d(X_N)=4$ by Lemma \ref{degree}.  If $X_N$ is $1$-regular then  $|{\rm Aut}(X_N)|=24p$ and so $|A|=24p^2$. Now $X$ is $1$-regular and it is a contradiction. Thus we may suppose that  $X_N$ is not $1$-regular. By Proposition \ref{Normal}, $A/C=N_{A}(N)/C_{A}(N) \leq \rm Aut(N)\cong \mathbb Z_{p-1}$. Now, we consider the following cases: \\
  \textbf{Case I:}  $N=C$.

  Therefore $A/N $ is a abelian group.  Also since   $K_{v}=1$ for any $v \in V(X)$,  we have $|N|=|K|$. Since $A/K$ acts transitively on $V(X_{N})$, we have $A/K$  is a regular group. Now by considering the arc-transitively  of $X$,  we get a contradiction.\\
  \textbf{Case II:}  $N < C$.

   Let $T/N $ be a minimal normal subgroup of $A/N$ and $T/N \leq C/N$.  Since $T \leq C$ and $N \leq T$, we have $N \leq Z(T)$. If $T/N $ is a non-solvable group then by Proposition \ref{4} and since  $p\geq 11$,  $T/N \cong$ PSL(2,17) or PSL(3,3). \\
   \textbf{Subcase I:}  $T=T'$.

   If $T=T'$ then then $T$ is a covering group of PSL(2, 17) or PSL(3,3). The Schur multiplier of  PSL(2, 17)   is $\mathbb{Z}_2$ and the Schur multiplier of PSL(3,3) is $1$. Since $N \cong  \mathbb{Z}_p$, it implies that $T=N \times$ PSL(2,17) or $T=N \times$ PSL(3,3). Then $T=T'=$ PSL(2,17)$'$ or  PSL(3,3)$'$ and hence $N=1$, a contradiction.\\
    \textbf{Subcase II:}  $T' <T$.

    Suppose that $T' <T$. Since  $T/N=T'N/N \cong$ PSL(2,17) or PSL(3,3), we have $T=T'N$. If $T' \cap N \neq 1$ then $T' \cap N = N $,   $N < T'$ and   hence $T=T'$, it is a contradiction. Then $T=T' \times N$ and hence $T/N\cong T' \cong$ PSL(2,17) or PSL(3,3). We know that $T'$ is characteristic in $T$ and hence normal in $A$. Consider the quotient graph $X_{T'}$. By Proposition \ref{reduction}, $X_{T'}$ has valency $2$ or $4$. If $X_{T'}$ has valency $4$, then $T'$ is semiregular and hence $T'$ is a solvable group, a contradiction. Assume that $d(X_{T'})=2$. By case (3) of Proposition \ref{reduction}, $X_{T'}$ is 1-transitive and hence $|A/N|= 2^{m+1}.3.p$, where $p\geq 11$. Then $T/N$ is solvable,  a contradiction.

Therefore  $T/N $ is a solvable group and $N < C$. Then   $T/N $  is an elementary abelian 2-,3- or p-group. Let $M:=T/N$ and $Y:=X_{N}$. By  by [\cite{gh10}, Theorem 1] $Y \cong C_{3p}[2K_1]$ with $|{\rm Aut}(Y)|=2^{3p+1}.3.p$.  Suppose that   $Y_{M}$ be the quotient graph of $Y$ corresponding to the orbits of $M$ on $V(Y)$. Assume that $K_{1}$ be the kernel of $\rm Aut(Y)$ acting on $V(Y)$. \\
\textbf{Subcase a:}  $M$ is a 2-group.

 Let  $M$ is a 2-group. Then $|V(Y_{M})|= 3p$. If  $Y_{M}$ has   valency $4$ then by [\cite{WX1},  Theorem 5], $Y_{M}\cong G(3p,4)$ and $|\rm Aut(Y_{M})|=12p$. Since $(K_{1})_{v}=1$ and $|K_{1}|=|M|=2$, $|\rm Aut(Y)/ K_{1}| \leqslant 12p$ and hence $|\rm Aut(Y)| \leq 48 p$. It is a contradiction because $p > 11$. Suppose now that $Y_{M}$ has   valency $2$. By [\cite{WX1},  Theorem 5], $Y_{M}\cong G(3p,1)$ and $|\rm Aut(Y_{M})|=6p$. Since $|K_{1}|\leq 8$, $|\rm Aut(Y)/ K_{1}| \leqslant 48p$ and hence $2^{3p+1}\leq 16$, a contradiction.\\
\textbf{Subcase b:}   $M$ is a 3-group.

Let  $M$ is a 3-group. By Lemma \ref{degree},  $Y_{M}$ has   valency $4$.  By [\cite{CO}, Table 1], $Y_{M}\cong G(2p,4)$ or G(2,p,2). Assume that  $Y_{M}\cong G(2p,4)$. Since $(K_{1})_{v}=1$, $|\rm Aut(Y)/ K_{1}| \leqslant 8p$ and hence $|\rm Aut(Y)| \leq 24 p$. It is a contradiction  because $p \geq 11$. Suppose that $Y_{M}\cong G(2,p,2)$. Then $|\rm Aut(Y)| \leq 2^{p+1} p$ and hence $3p+1 \leq p+1$, a contradiction.
\\
\textbf{Subcase c:}   $M$ is a p-group.

 Let  $M$ is a $p$-group. Then $M_{v}=1$ for $v \in Y$ and hence $Y_{M}$ has   valency $4$. By sage software \cite{sage}  $Y_{M}$ is octahedron graph $O_6$. Since $|\rm Aut(Y_{M})|=48 $ and $|K_{1}|=p$, we have $|\rm Aut(Y)/ K_{1}| \leqslant 48$. We get a contradiction because $p \geq 11$.

	Now suppose that $N \cong \mathbb{Z}_p \times \mathbb{Z}_p$. If $N$ has an orbit of length $p$ then its stabilizer has an element of order $p$ and by Proposition \ref{4-val-sym}, a contradiction.   Therefore all orbits of $N$ has length $p^2$ and  $|X_N|=6$. Also since $N$ is a normal subgroup of $A$ and acts semiregulary on $V(X)$ we conclude that $X$ is a $\mathbb{Z}_p \times \mathbb{Z}_p$-regular cover of octahedron graph $O_6$. By  [\cite{Koh}, Table1], there is no   $\mathbb{Z}_p \times \mathbb{Z}_p$  cover of octahedron graph $O_6$.

\end{proof}

\end{document}